\documentclass[11pt]{article}

\usepackage[dvipdfmx]{graphicx}
\usepackage{wrapfig}

\usepackage[fleqn]{amsmath}
\usepackage[dvipdfmx]{graphicx,color}
\usepackage{amsthm}
\usepackage{amsmath}
\usepackage{amssymb}
\usepackage{amsbsy}
\usepackage{amsfonts}
\usepackage{mathrsfs}
\usepackage[all]{xy}
\usepackage{amstext}
\usepackage{amscd}
\usepackage[dvips]{epsfig}
\usepackage{psfrag}
\usepackage{enumerate}
\usepackage{flafter}
\usepackage{mathtools}
\usepackage{empheq}

\allowdisplaybreaks

\newtheoremstyle{mystyle}
    {}
    {}%
    {\itshape}%
    {}%
    {\bfseries}%
    {.}%
    { }%
    {\thmname{#1}\thmnumber{ #2}\thmnote{ #3}}
\theoremstyle{mystyle}
\newtheorem{thm}{Theorem}[section]
\newtheorem{prop}[thm]{Proposition}
\newtheorem{cor}[thm]{Corollary}
\newtheorem{lem}[thm]{Lemma}
\newtheorem{conj}[thm]{Conjecture}
\theoremstyle{definition}

\newtheorem{rem}[thm]{Remark}

\def\4{\mathop{4 \mathrm{f}}\nolimits}

\def\12{\mathop{X_{12}}\nolimits}

\def\rank{\mathop{\mathrm{rank}}\nolimits}
\def\dim{\mathop{\mathrm{dim}}\nolimits}

\def\Im{\mathop{\mathrm{Im}}\nolimits}

\def\Hom{\mathop{\mathrm{Hom}}\nolimits}
\def\Aut{\mathop{\mathrm{Aut}}\nolimits}

\def\Tor{\mathop{\mathrm{Tor}}\nolimits}
\def\sgn{\mathop{\mathrm{sgn}}\nolimits}

\def\SL{\mathop{\mathrm{SL}}\nolimits}

\newcommand{\mf}[1]{{\mathfrak{#1}}}
\newcommand{\mb}[1]{{\mathbf{#1}}}

\newcommand{\mca}[1]{{\mathcal{#1}}}

\newcommand{\ra}{\rightarrow}
\newcommand{\Q}{{\Bbb Q}}
\newcommand{\R}{{\Bbb R}}
\newcommand{\Z}{{\Bbb Z}}
\newcommand{\N}{{\Bbb N}}
\newcommand{\F}{{\Bbb F}}
\newcommand{\C}{{\Bbb C}}

\begin{document}
\title{
Adjoint Reidemeister torsions of some 3-manifolds obtained by Dehn surgeries
}
\author{Naoko Wakijo\footnote{
E-mail address: {\tt wakijo.n.aa@m.titech.ac.jp}}}
\date{}
\pagestyle{plain}

\maketitle

\begin{center}
\normalsize

 \footnote{AMS 2010 Mathematics Subject; Primary: 57K30, 57K31, 57Q10, Secondary: 57M05}
\end{center}

\begin{abstract}
We determine the adjoint Reidemeister torsion of a $3$-manifold obtained by some Dehn surgery along $K$, where $K$ is either the figure-eight knot or the $5_2$-knot.
As in a vanishing conjecture \cite{BGP-JHEP20, GKP-JHEP20,GKY-ATMP21},
we consider a similar conjecture and show that the conjecture holds for the 3-manifold. 
\end{abstract}
\section{Introduction}
Let $\mf{g}$ be the Lie algebra of a semisimple complex Lie group $G$, and $M$ be a connected compact oriented manifold.
Let $R^{\rm irr}_G(M)$ be the (irreducible) character variety, that is, the set of conjugacy classes of irreducible representations $\pi_1(M) \rightarrow G$. 
Given a homomorphism $ \varphi : \pi_1(M) \ra G $, we can define the {\it adjoint (Reidemeister) torsion} $\tau_\varphi(M)$ {under a mild assumption}, which lies in $\C^\times$ and is determined by the conjugacy class of $\varphi$; see \cite{Turaev01} or Section \ref{secReview} for details.
When $ \dim M =2$, the torsion plays an interesting role as a volume form on the space $R^{\rm irr}_G(M)$; see \cite{PortiMAMS97, WittenCMP91}.
In addition, if $M$ is 3-dimensional and $G= \SL_2(\C)$,
some attitudes of the torsions in $R^{\rm irr}_G(M)$ are physically observed from the viewpoint of a 3D-3D correspondence,
and some conjectures on the torsions are mathematically proposed {in} \cite{BGP-JHEP20,GKP-JHEP20,GKY-ATMP21}.

{For instance, with reference to \cite{GKY-ATMP21}, the conjecture can be roughly described as follows.
Suppose that $\dim M=3$ and $M$ has a tori-boundary.
For $z \in \C$, introduce a finite subset ``$\mathrm{tr}_{\gamma}^{-1}(z)$" of $R^{\rm irr}_G(M)$ which is defined from a boundary condition,
and discusse the sum of the $n$-th powers of the twice torsions, that is, $\sum_{\varphi \in \mathrm{tr}_{\gamma}^{-1}(z) }(2\tau_{\varphi}(M ))^n \in \C$ for $n \in \Z $ with $n \geq -1$.
Then, the studies in \cite{BGP-JHEP20,GKP-JHEP20,GKY-ATMP21} suggest that the sum lies in $\Z$ and, that if $M$ is hyperbolic and $n=-1$, then the sum is zero.} This conjecture is sometimes called {\it the vanishing identity}; see \cite{porti2021adjoint,tran2023adjoint,Yoon22} and references therein for supporting evidence of this conjecture.

In this paper, we focus on the adjoint torsions in the case where $\dim{M}=3$ and $M$ has no boundary.
According to {\cite{BGP-JHEP20,GKP-JHEP20}},
it is seemingly reasonable to consider the following conjecture: 
\begin{conj}[{{\rm{\cite{BGP-JHEP20,GKP-JHEP20}}}}]\label{conj1}
Take $n \in \Z$ with $n \geq -1$.
Suppose that $M$ is {a} closed $3$-manifold, and the set $R^{\rm irr}_G(M)$ is finite.
Then, the following sum lies in the ring of integers $\Z$:
\begin{equation}\label{key1}\sum_{\varphi \in R^{\rm irr}_G(M)} ( 2 \tau_{\varphi}(M))^{n} . \end{equation}
Furthermore, if $G=\SL_2(\C)$, $M$ is a hyperbolic 3-manifold, and $n=-1$, then the sum is zero.
\end{conj}
\noindent In \cite{Shawn23}, when $G=\SL_2(\mathbb{C})$, {the adjoint torsion of certain Seifert 3-manifolds and torus bundles are explicitly computed}; thus, we can easily check the conjecture for the non-hyperbolic 3-manifolds.

In contrast, this paper provides supporting evidence on Conjecture \ref{conj1} in hyperbolic cases.
For $p/q \in \Q$ and a knot $K$ in $S^3$, let $S^3_{p/q}(K)$ be the closed $3$-manifold obtained by $(p/q)$-Dehn surgery on $K$.
\begin{thm}\label{thm1}
Let $G$ be $\SL_2(\C)$, and $K=4_1$ be the figure-eight knot. Let $n=-1$. Then, for any integers $p$ and $q\neq0$, Conjecture \ref{conj1} is true when 
$M=S^3_{p/1}(4_1)$ and $M=S^3_{1/q}(4_1)$.
\end{thm}
\noindent
{We similarly discuss whether Conjecture \ref{conj1} is true for $M=S^3_{1/q}{(K)}$ when the knot $K$ is the $5_2$-knot; see Section \ref{sec52}.}

The outline of the proof is as follows. 
While some computations of the adjoint torsions of 3-manifolds with boundary are established (see, e.g., \cite{Dubois09,tran2023adjoint,Yoon22}),
this paper employs
a procedure of computing the adjoint torsions of closed 3-manifolds, which is established in \cite{NW1},
and we determine all the adjoint torsion (Theorems \ref{thm2-1} and \ref{thm2-2}).
As in the previous proof of the above supporting evidence,
we apply Jacobi's residue theorem (see Lemma \ref{lem5}) to the sum \eqref{key1} and demonstrate Theorem \ref{thm1}. 
Since it is complicated to check the condition for applying the residue theorem, we need some careful discussion (see Sections \ref{seclem}--\ref{subsecProof})\footnote{As a private communication with S. Yoon, he tells us another proof of Conjecture \ref{conj1} with $M= S^3_{p/q}(K)$ in generic condition. Here, we emphasize that, while the condition does not contain the case $(p,q)=(4m,1)$ for some $m \in \Z$, Theorem \ref{thm1} deals with all $p.$
}.
Finally, in Section \ref{secother}, we also discuss the conjecture with $n >0$, and
see that some properties are needed to be addressed in future studies.
Here, we show the $2^{2n+1}$-multiple of the conjecture with $M=S^3_{2m/1}(4_1)$; see Proposition \ref{prop4}.

\section{Review; the adjoint Reidemeister torsion}\label{secReview}
After reviewing algebraic torsions in Section \ref{subsecTor}, we briefly recall the definition of the adjoint Reidemeister torsion in Section \ref{subsecAdj}.
{We note that our definition of the adjoint torsion is of sign-refined type.}
Section \ref{subsecPres} explains cellular complexes of $M$.
Throughout this paper, we assume that any basis of a vector space is ordered.

\subsection{Algebraic torsion of a cochain complex}\label{subsecTor}

Let $C^*$ be a bounded cochain complex consisting of finite dimensional vector spaces over a commutative field $\F$, that is,
\[
C^*=(0\ra C^0\xrightarrow{\delta^0}C^{1} \xrightarrow{\delta^{1}}\cdots \xrightarrow{\delta^{m-1}} C^m\ra 0).
\]
Let $H^i=H^i(C^*)$ be the $i$-th cohomology group.
Choose a basis $\mb{c}^i$ of $C^i$ and a basis $\mb{h}^i$ of $H^i$.
The {\it Reidemeister torsion} $\Tor(C^*,\mb{c}^*,\mb{h}^*)$ is defined as follows.

Let $\widetilde{\mb{h}}^i\subset C^i$ be a representative {cocycle} of $\mb{h}^i$ in $C^i$. Let $\mb{b}^i$ be a tuple of vectors in $C^i$ such that $\delta^i(\mb{b}^i)$ is a basis of $B^{i+1}=\Im{\delta^i}$. Then the union of the sequences of the vectors $\delta^{i-1}(\mb{b}^{i-1})\widetilde{\mb{h}}^i \mb{b}^i$ gives a basis of $C^i$.
We write $[\delta^{i-1}(\mb{b}^{i-1})\widetilde{\mb{h}}^i \mb{b}^i/\mb{c}^i] \in \F^\times =\F\setminus \{0\}$ for the determinant of the transition matrix that takes $\mb{c}^i$ to $\delta^{i-1}(\mb{b}^{i-1})\widetilde{\mb{h}}^i \mb{b}^i$.
Let $|C^*|$ be $\sum_{i=0}^{m}\alpha_i(C^*)\beta_i(C^*)$ where $\alpha_i(C^*)\coloneqq\sum_{j=0}^{i}\dim{C^j}$ and $\beta_i(C^*)\coloneqq\sum_{j=0}^i\dim{H^j}$.
Let $\mb{c}^*$ be $(\mb{c}^0,\dots,\mb{c}^{m})$ for $C^*$ and $\mb{h}^*$ be $(\mb{h}^0,\dots,\mb{h}^{m})$ for $H^*$.
Then, the torsion is defined to be the alternating product of the form
\[\Tor(C^*,\mb{c}^*,\mb{h}^*)\coloneqq(-1)^{|C^*|}\prod_{i=0}^{m} [\delta^{i-1}(\mb{b}^{i-1})\widetilde{\mb{h}}^i \mb{b}^i/\mb{c}^i]^{(-1)^{i+1}}\in \F^\times.\]
It is known that the torsion $\Tor(C^*,\mb{c}^*,\mb{h}^*)$ does not depend on the choices of $\widetilde{\mb{h}}^i$ and $\mb{b}^i$, but depends only on $\mb{c}^*$ and $\mb{h}^*$.
We refer to 
\cite{MilnorBAMS66,Turaev01} for the details.
Note that, if $C^*$ is acyclic (i.e.{,} $H^*(C^*)=0$), then the torsion $\Tor(C^*,\mb{c}^*,\mb{h}^*)$ is usually denoted {by} $\Tor(C^*,\mb{c}^*)$.

\begin{rem}
In \cite{MilnorBAMS66} and \cite{Turaev01}, the torsion was defined from a chain complex; however, for convenience of computation, we define the torsion from a cochain complex in this paper.
\end{rem}

\subsection{Adjoint Reidemeister torsion of a $3$-manifold}\label{subsecAdj}
Let $M$ be a connected oriented closed $3$-manifold, and let $G$ be a semisimple Lie group with Lie algebra $\mf{g}$.
Let $\varphi:\pi_1(M)\rightarrow G$ be a representation, that is, a group homomorphism.
Suppose that $G$ injects $\SL_n(\C)$ for some $n\in\N$.

First, we introduce the cochain complex.
Choose a finite cellular decomposition of $M$ and consider the universal covering space $\widetilde{M}$. We can canonically obtain a cellular structure of $\widetilde{M}$ as a lift of the decomposition of $M$, and define the cellular complex $(C_*(\widetilde{M};\Z),\partial_*)$.
We regard the covering transformation of $M$ as a left action of $\pi_1(M)$ on $\widetilde{M}$, and naturally regard $C_*(\widetilde{M};\Z)$ as a left $\Z[\pi_1(M)]$-module.
Since $\mf{g}$ is a left $\Z[\pi_1(M)]$-module via the composite of $\varphi $ and the adjoint action $G\rightarrow \Aut(\mf{g})$, we have the cochain complex of the form
\[(C^*_\varphi({M};\mf{g}),\delta^*)\coloneqq(\Hom_{\Z[\pi_1(M)]}(C_*(\widetilde{M};\Z),\mf{g}),\delta^*)\]
where $\delta^i$ is defined by $\delta^i(f)=f\circ \partial_{i{+1}}$.

Next, we define an ordered basis of $C^i_\varphi({M};\mf{g})$.
Let $\mb{c}_i=(c_{i,1},c_{i,2},\ldots,c_{i,\rank_\Z{C_i({M};\Z)}})$ be a basis of $C_i({M};\Z)$ derived from the $i$-cells.
Then, $\widetilde{\mb{c}}_i=(\widetilde{c}_{i,1},\widetilde{c}_{i,2},\ldots,\widetilde{c}_{i,\rank_\Z{C_i({M};\Z)}})$ is a basis of the free $\Z[\pi_1(M)]$-module $C_i(\widetilde{M};\Z)$. Here, $\widetilde{c}_{i,j}$ is a lift of $c_{i,j}$ to $\widetilde{M}$.
Since $\mf{g}$ is semisimple, the {Killing} form $B$
is non-degenerate, and we can fix an ordered basis $\mca{B}=(e_1,e_2,\ldots,e_{\dim{\mf{g}}})$ of $\mf{g}$ that is orthogonal with respect to $B$. 
Let $c_{i,j}^k\in C^i_\varphi({M};\mf{g})$ be a $\Z[\pi_1(M)]$-homomorphism defined by $c_{i,j}^k(\widetilde{c}_{i,\ell})=\delta_{j,\ell}e_k\in\mf{g}$ for any $i\in\{0,1,2,3\}$, $j,\ell\in\{1,2,\ldots,\rank_\Z{C_i({M};\Z)}\}$, and $k\in\{1,2,\ldots,\dim{\mf{g}}\}$.
Here, $\delta_{j,\ell}$ is the Kronecker delta.
Then the tuple
\[\mb{c}^i=(c_{i,1}^1,c_{i,1}^2,\ldots,c_{i,1}^{\dim{\mf{g}}},c_{i,2}^1,c_{i,2}^2,\ldots,c_{i,2}^{\dim{\mf{g}}},\ldots,c_{i,\rank_\Z{C_i(M;\Z)}}^1,c_{i,\rank_\Z{C_i(M;\Z)}}^2,\ldots,c_{i,\rank_\Z{C_i(M;\Z)}}^{\dim{\mf{g}}})\]
provides an ordered basis of $C^i_\varphi(M;\mf{g})$ as desired.

We next consider the cellular cochain complex $C^*(M;\R)$ with the real coefficient.
Let $c^i_j:C_i(M;\Z)\ra \R$ be a homomorphism defined by $c^i_j(c_{i,k})=\delta_{j,k}$ for any $i\in\{0,1,2,3\}$ and $j,k\in\{1,2,\ldots,\rank_\Z{C_i({M};\Z)}\}$.
Then, $\mb{c}^i_\R=(c^i_1,\ldots,c^i_{\rank_\Z{C_i(M;\Z)}})$ is a basis of $C^i(M;\R)$.
{By} Poincar\'e duality, we can naturally fix a {\it homology orientation} $\sigma_M$ of $H^*(M;\R)= \bigoplus_{i = 0}^3 H^i(M;\mathbb{R})$.
Let {$\mb{h}^*_\R
$} be a basis of $H^*(M;\R)$ such that the exterior product of $\mb{h}^*_\R$ coincides with $\sigma_M$.
The Reidemeister torsion of $C^*(M;\R)$ associated with $\mb{c}^*_\R$ and $\mb{h}^*_\R$ lies in $\R^\times$. Therefore, we can define the sign
\[\tau_M\coloneqq\sgn (\Tor(C^*(M;\R),\mb{c}^*_\R,\mb{h}^*_\R))\in \{\pm1\}.\]
Then, the {\it adjoint Reidemeister torsion} of $M$ associated with $\varphi$ is defined to be
\[\tau_{\varphi}(M)\coloneqq(\tau_M)^{\dim \mf{g}} \cdot \Tor(C^*_\varphi({M};\mf{g}),\mb{c}^*) \in \C^\times,\]
if $C^*_\varphi({M};\mf{g})$ is acyclic. If $C^*_\varphi({M};\mf{g})$ is not acyclic, then we define $\tau_{\varphi}(M)=1$.
As is known \cite{Dubois09,PortiMAMS97}, the definition of $\tau_\varphi(M)$ does not depend on the choices of the orthogonal basis $\mca{B}$, finite cellular decompositions of $M$, $\widetilde{\mb{c}_i}$, and $\mb{h}^i_\R$, but depends only on $M$ and the conjugacy class of $\varphi$.

Finally, we give a sufficient condition for the acyclicity, which might be known:
\begin{lem}\label{lem4}
As in Conjecture \ref{conj1}, assume that $R^{\rm irr}_G(M)$ is of finite order.
Then, for any irreducible representation $ \varphi: \pi_1(M) \ra G$,
the associated cohomology {$H^{*}_\varphi( M;\mathfrak{g})$} is acyclic.
\end{lem}
\begin{proof} Since it is classically known {\cite{Weil64}} that
the first cohomology $H^{1}_\varphi( M;\mathfrak{g})$ is identified with the cotangent space of the variety $R^{\rm irr}_G(M)$, it vanishes by assumption; by Poincar\'{e} duality, the second one does.
Meanwhile, by definition, the zeroth cohomology $H^{0}_\varphi( M;\mathfrak{g})$ equals the invariant part $\{ a \in \mathfrak{g} \mid a \cdot \varphi (g) = a \mathrm{ \ for \ any \ } g \in \pi_1(M) \}$, which is zero by the irreducibility.
Hence, the third one also vanishes by Poincar\'{e} duality again.
\end{proof}

\subsection{Presentations of the cellular complexes of $M$}\label{subsecPres}

From now on, we assume that $G=\SL_2(\C)$ and $M$ is one of
$S^3_{p/1}(4_1)$ and $S^3_{1/q}(4_1)$ for some integers $p$ and $q\neq0$ as in Theorem \ref{thm1}.
According to \cite{NosakaJMSCT22}, {group presentations of $\pi_1(M)$ are given as follows}:
\begin{equation}\label{Eq:pres}
\begin{split}
\pi_1(S^3_{p/1}(4_1))&\cong\langle x_1,x_2,\mf{m}\, |\,
\mf{m}x_1x_2\mf{m}^{-1}x_1^{-1},
\mf{m}x_2x_1x_2\mf{m}^{-1}x_2^{-1},
[x_1,x_2]\mf{m}^p
\rangle, \\
\pi_1(S^3_{1/q}(4_1))&\cong\langle x_1,x_2,\mf{m},\mf{m}'\, |\,
\mf{m}x_1x_2\mf{m}^{-1}x_1^{-1},
\mf{m}x_2x_1x_2\mf{m}^{-1}x_2^{-1},
\mf{m}[x_1,x_2]^q,
\mf{m}'[x_1,x_2]^{-1}
\rangle
.
\end{split}
\end{equation}
Here, $[x,y]$ is $xyx^{-1}y^{-1}$.
Let $g$ be the number of generators of the group presentation above.
Replace $\mf{m}$ by $x_3$, $\mf{m}'$ by $x_4$, and let $r_i$ denote the $i$-th relator in \eqref{Eq:pres}.
Under the identifications
$C^i_\varphi({M};\mf{g})=\Hom_{\Z[\pi_1(M)]}(C_i(\widetilde{M};\Z), \mf{g}) = \Hom_{\mf{g}}(\mf{g}^{\rank _\Z C_i(M;\Z)},\mf{g})$,
the cochain complex $(C^*_\varphi({M};\mf{g}),\delta^*)$ is {isomorphic to} the dual of the following chain complex:
\begin{equation}\label{Eq:chain}
0\ra \mf{g}\xrightarrow{\delta^3}\mf{g}^{g} \xrightarrow{\delta^2} \mf{g}^{g}\xrightarrow{\delta^1} \mf{g}\ra 0.
\end{equation}

We now describe the differentials $\delta^*$ in detail.
Let $F$ and $P$ be the free groups $\langle x_1, \dots, x_g\ | \ \rangle$ and $ \langle \rho_1,\ldots,\rho_g \ | \ \rangle$, respectively.
We define the homomorphism $\psi: P *F \ra F$ by setting $\psi (\rho_j)=r_j$ and $\psi(x_i)=x_i.$
Let $\mu$ denote the natural surjection from $F$ to $\pi_1(M)$.
According to \cite[\S3.1]{NosakaJMSCT22}, we can describe $\delta^*$ by the words of the presentations \eqref{Eq:pres} as follows: let $W\in P*F$ be
{\[\rho_1\cdot x_1\rho_2 x_1^{-1}\cdot (x_1x_2x_1^{-1})\rho_1^{-1}(x_1x_2x_1^{-1})^{-1}\cdot ([x_1,x_2])\rho_2^{-1}([x_1,x_2])^{-1}\cdot \rho_3 \cdot \mf{m}\rho_3^{-1}\mf{m}^{-1},\]}
if $M=S^3_{p/1}(4_1)$. Let $W \in P*F$ be
\[\rho_1\cdot x_1\rho_2 x_1^{-1}\cdot (x_1x_2x_1^{-1})\rho_1^{-1}(x_1x_2x_1^{-1})^{-1}\cdot ([x_1,x_2])\rho_2^{-1}([x_1,x_2])^{-1}\cdot \rho_4^{-1}\cdot \mf{m}'\rho_3\mf{m}'^{-1}\cdot \rho_4 \cdot \rho_3^{-1},\]
if $M=S^3_{1/q}(4_1)$.
Then, each $\delta^*$ can be written as the matrices
\begin{equation}\label{Eq:delta}
\delta^1=\left(1-x_j\right)_{j=1,\ldots,g},\,\,\,
\delta^2=\left({\small \frac{\partial r_j}{\partial x_i}}\right)_{i,j=1,\ldots,g},\,\,\,
\delta^3=\mu \circ \psi\left({\small \frac{\partial W}{\partial \rho_i}}\right)_{i=1,\ldots,g},
\end{equation}
where $\frac{\partial *}{\partial *}$ is {\it Fox derivative}, see \cite[\S 16]{Turaev01} for the definition.
Although each entry of the matrices is described in $\Z[\pi_1(M)]$, we regard the entry as an automorphism of $\mf{g}$ via the adjoint action.

\section{Proof of Theorem \ref{thm1} }\label{secProof}
The purpose of this section is to show the proof of Theorem \ref{thm1}.
First, Section \ref{subsecPre} determines the torsion with respect to every irreducible representation. Next, Section \ref{seclem} establishes two key lemmas, and Section \ref{subsecProof} completes the proof.
Throughout this section, $E_2$ means the $(2\times 2)$-identity matrix, and we let $G$ be $\SL_2(\C)$.

\subsection{Preliminary}\label{subsecPre}
\begin{wrapfigure}[9]{r}[10mm]{40mm}\vspace{-10pt}
\begin{picture}(100,100)
\put(0,-2){\includegraphics[scale=0.25]{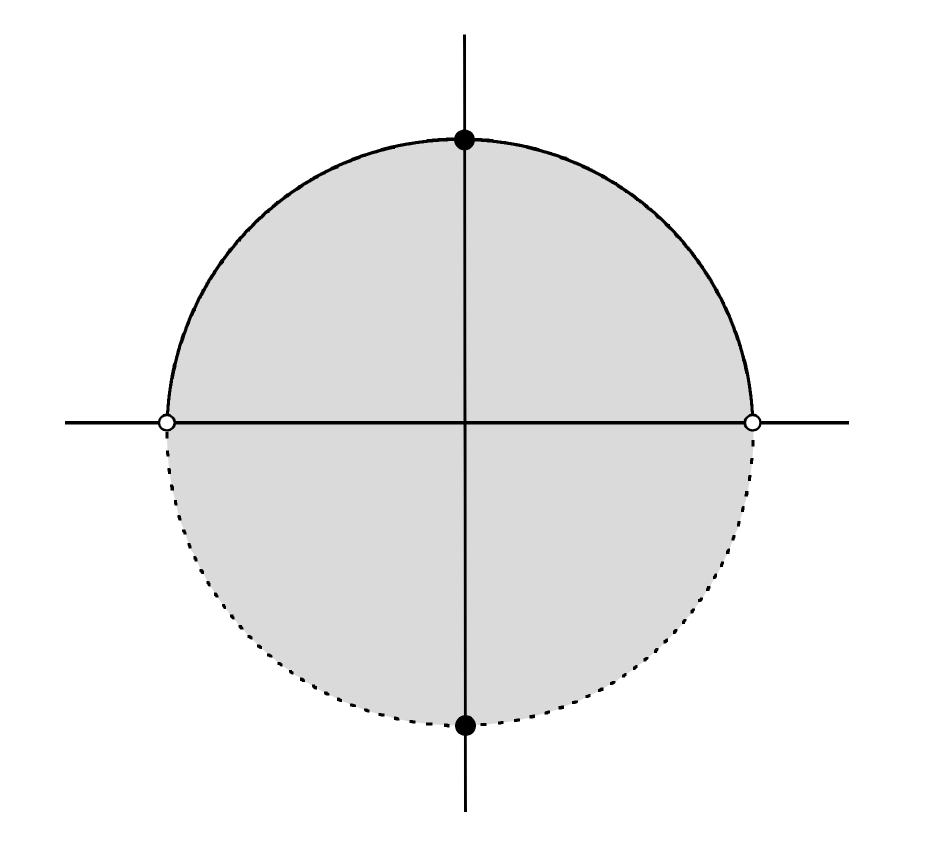}}
\put(10,90){\fbox{$\C$}}
\put(100,38){Re}
\put(92,40){{\tiny $1$}}
\put(10,40){{\tiny $-1$}}
\put(38,90){Im}
\put(58,85){{\tiny $\sqrt{-1}$}}
\put(58,7){{\tiny $-\sqrt{-1}$}}
\put(60,35){$D$}

\end{picture}
\vspace{-15pt}\caption{$D\subset \C$}\label{fig1}
\end{wrapfigure}

To state {Propositions \ref{prop1-1}, \ref{prop1-2} and Theorems \ref{thm2-1}, \ref{thm2-2}}  below, let us consider a domain $D$ in $\C$ of the form
\[D\coloneqq \{a\in\C\,\,|\,\,|a|<1\} \cup \{a\in \C\,\,\,|\,\,\,\Im(a)>0,|a|=1\}\cup\{-\sqrt{-1}\},\]
as in Figure \ref{fig1}, and
define the {Laurent} polynomial $Q_M(x)\in \Z[x{,x^{-1}}]$ by setting
\[Q_{M}(x)\coloneqq\left\{
\begin{array}{ll}
1-x^{p-4}+x^{p-2}+2x^p+x^{p+2}-x^{p+4}+x^{2p} ,& \text{if }M=S^3_{p/1}(4_1),\\
1-x^{2q}-x^{4q-1}-2x^{4q}-x^{4q+1}-x^{6q}+x^{8q} ,& \text{if }M=S^3_{1/q}(4_1)
.\end{array}
\right.
\]
Let $Q_M^{-1}(0){\in \C}$ denote the zero set of the {Laurent} polynomial $Q_M$.

\begin{prop}\label{prop1-1}
Let $M=S^3_{p/1}(4_1)$ for some integer $p$.
If $p\neq 0$, then there is a bijection $\Phi_{M}:R^{\rm irr}_G(M)\ra Q_M^{-1}(0)\cap D$.
Here, for $[\varphi]\in R^{\rm irr}_G(M)$, we define
\begin{equation}\label{Eq:Phi}
\Phi_M([\varphi])\coloneqq ({\mathrm {The\,\, eigenvalue \,\,of\,\, }}\varphi(\mf{m}) {\mathrm{ \,\,that\,\, lies\,\, in\,\, } }D\setminus\{\pm \sqrt{-1}\}),\end{equation}
when the eigenvalues of $\varphi(\mf{m})$ are not $\pm \sqrt{-1}$.
If $\varepsilon \sqrt{-1}\in Q_M^{-1}(0)$ for some $\varepsilon \in \{\pm1\}$, $\Phi_M^{-1}(\varepsilon \sqrt{-1})$ {is a conjugacy class with a representation $\varphi$ defined by}
\begin{equation}\label{Eq:var}\varphi(\mf{m})=\left(\begin{array}{cc}
\varepsilon \sqrt{-1} & 0 \\
0 & -\varepsilon \sqrt{-1}\\
\end{array}
\right),\,\,\,\,
\varphi(x_1)=\left(\begin{array}{cc}
\frac{1}{4} (-1+\varepsilon\sqrt{5}) & 1 \\
\frac{1}{8} (-5-\varepsilon\sqrt{5}) & \frac{1}{4} (-1+\varepsilon\sqrt{5}) \\
\end{array}
\right).
\end{equation}
If $p=0$, then there is a bijection
$\Phi_{M}:R^{\rm irr}_{G}(M)\ra \{\pm \sqrt{-1} ,\pm (1-\sqrt{5})/2\}$.

\end{prop}

\begin{prop}\label{prop1-2}
Let $M=S^3_{1/q}(4_1)$ for some integer $q\neq 0$.
Then, there is a bijection $\Phi_{M}:R^{\rm irr}_G(M)\ra Q_M^{-1}(0)\cap D$.
{Here, $\Phi_M$ is defined by \eqref{Eq:Phi} as in Proposition \ref{prop1-1}. Note that, since $\pm\sqrt{-1}\notin Q_M^{-1}(0)$ for $M=S^3_{1/q}(4_1)$, \eqref{Eq:var} can be excluded from the definition in this case.}
\end{prop}

\begin{proof}[{Proof of Proposition \ref{prop1-1}}]
Let $p\neq 0$.
For an irreducible representation $\varphi:\pi_1(M)\ra \SL_2(\C)$, take $x,y,z,w\in\C$ so that $\varphi(x_1)=\left( {\small \begin{array}{cc}x&y\\z&w\\ \end{array}}\right)$ and $xw-yz=1$. 
We first claim that $\varphi(\mf{m})$ is diagonalizable.
In fact, if not so, we may suppose
$\varphi(\mf{m})=\left({\small \begin{array}{cc}\eta&b\\0&\eta\\ \end{array}} \right)$ for some $b\in \C^\times$ and $\eta\in \{\pm1\}$.
Since $\varphi(r_1)=E_2$, we have
\begin{equation}\label{Eq:x2}
\varphi(x_2)=\varphi(x_1){^{-1}}\varphi(\mf{m})^{-1}\varphi(x_1)\varphi(\mf{m})=
\left({\small
\begin{array}{cc}
1-\eta b w z & - b \left(b w z+\eta w^2-\eta\right) \\
\eta b z^2 & b^2 z^2+\eta b w z+1 \\
\end{array}}
\right).
\end{equation}
It follows from
\eqref{Eq:x2}
that
\begin{equation}\label{Eq:r3}\begin{split}
\varphi(r_3)&=\varphi(x_1)\varphi(x_2)\varphi(x_1)^{-1}\varphi(x_2)^{-1}\varphi(\mf{m})^p\\
&=\eta^p{\small
\left(
\begin{array}{cc}
b^4 z^4+\eta b^3 w z^3-b^2 z^2 {\left( x^2+xw-3\right)} 
+\eta b z (w-x)+1 & * \\
-\eta b^3 z^4-b^2 z^3 (w+x)-\eta 2 b z^2 & {*}
 \\
\end{array}
\right).}
\end{split}\end{equation}
Then, the condition $\varphi(r_3)=E_2$ and $b\neq 0$ leads to $z=0$.
{In fact, if $z\neq 0$, the (2,1)-entries of $\varphi(r_3)=E_2$ yields $x=-\eta 2 b^{-1}z^{-1}-w-\eta b z$ by \eqref{Eq:r3}. Thus, the (1,1)-entry of \eqref{Eq:r3} equals $-1$, resulting in a contradiction.}

By substituting $z=0$ into $\varphi(r_2)$, we obtain
\[\begin{split}
E_2=\varphi(r_2)&=\varphi(\mf{m})\varphi(x_2)\varphi(x_1)\varphi(x_2)\varphi(\mf{m})^{-1}\varphi(x_2)^{-1}
=\left({\small
\begin{array}{cc}
x & y-\eta b \left(w^3-2 w+x\right) \\
0 & w \\
\end{array}}
\right).
\end{split}
\]
Thus, $x=w=1$ and $y=0$; therefore, $\varphi(x_1)$ and $\varphi(x_2)$ are upper triangular matrices, which leads to a contradiction to the irreducibility.

By the above claim, we may suppose $\varphi(\mf{m})=\left({\small \begin{array}{cc}a & 0 \\ 0&a^{-1}\\ \end{array}}\right)$ for some $a\in D{\setminus \{0\}}$.
Since we consider $\varphi$ up to conjugacy, we may suppose $y=1$. 
Thus, $z=xw-1$.
Since $\varphi(r_1)=\varphi(r_2)=\varphi(r_3)=E_2$,
with the help of a computer program of Mathematica, we have
\begin{equation}\label{Eq:xzw}
\begin{split}
x&={\small \frac{1+a^2-a^4+\eta(1-2 a^2-a^4-2 a^6+a^8)^{1/2}}{2(1- a^2)}},\\
z&={\small -\frac{1-3 a^2+a^4+\eta(1-2 a^2-a^4-2 a^6+a^8)^{1/2}}{2 (a^2-1)^2 }},\\
w&={\small \frac{-1+a^2+a^4+\eta(1-2 a^2-a^4-2 a^6+a^8)^{1/2}}{2 a^2 (a^2-1)}},\\
\end{split}
\end{equation}
and $Q_M(a)=0$ when $a\neq \pm \sqrt{-1}$.
Here, we fix a branch of the $1/2$-th power on $\C^\times\setminus \R$,
and define the signs $\eta\in \{\pm1\}$ by setting
\begin{equation*}
\eta=
\begin{cases}
+1, & {\mathrm{ if } }\,\, -1+a^2+2a^4+a^6-a^8+2 a^{p+4}=(a^4-1)(1-2 a^2-a^4-2 a^6+a^8)^{1/2},\\
-1, & {\mathrm{ if } } \,\,-1+a^2+2a^4+a^6-a^8+2 a^{p+4}=-(a^4-1)(1-2 a^2-a^4-2 a^6+a^8)^{1/2}.\\
\end{cases}
\end{equation*}
When $a=\varepsilon \sqrt{-1}$ for some $\varepsilon \in \{\pm1\}$, we have
{\begin{equation}\label{Eq:xzw2}
x= \frac{-1+\varepsilon \sqrt{5}}{4},\,\,\,\,\, z=\frac{-5-\varepsilon\sqrt{5}}{8},\,\,\,\,\, w= \frac{-1+\varepsilon\sqrt{5}}{4}\end{equation}}by the condition $\varphi(r_1)=\varphi(r_2)=\varphi(r_3)=E_2$.
In summary, the map $\Phi_M$ is well-defined and injective.
Finally, we can easily show the surjectivity of {$\Phi_M$} by following the reverse process of the above calculation.

{In the remaining case of $p=0$, define $\Phi_M$ as follows:
for each $\varepsilon \in\{\pm1\}$, let $\Phi_M^{-1}(\varepsilon\sqrt{-1})$ be a representation $\varphi$ defined by \eqref{Eq:var}. For $\varepsilon' \in\{\pm1\}$, let $\Phi_M^{-1}(\varepsilon' (1-\sqrt{5})/2)$   has a representation $\varphi$ defined by
\[\varphi(\mf{m})=
\left(\begin{array}{cc}
\varepsilon' \frac{1-\sqrt{5}}{2} & 0 \\
0 & \varepsilon' \frac{-1-\sqrt{5}}{2}\\
\end{array}
\right),\,\,\,\,
\varphi(x_1)=\left(\begin{array}{cc}
1 & 1 \\
0 & 1\\
\end{array}
\right),
\]
respectively. Then, by $\varphi(r_1)=\varphi(r_2)=\varphi(r_3)=E_2$, we can show the well-definedness and injectivity of $\Phi_M$ as in the case of $p\neq0$. By follwing the reverse process, we can check the surjectivity of $\Phi_M$ as well.}
\end{proof}

\begin{proof}[{Proof of Proposition \ref{prop1-2}}]
{It can be proved in the same manner as Proposition \ref{prop1-1}.
In this case, instead of \eqref{Eq:xzw}, we have 
\begin{equation}\label{Eq:xzwq}
\begin{split}
x=\frac{a^{-2 q} \left(2 a^{6 q}+2 a^{4 q+1}\right)}{2 \left(a^{2 q}-1\right)^2 \left(a^{2
   q}+1\right)}, \ z=-\frac{4 a^{2 q}+2 a^{4 q}-2 a^{6 q}+2 a^{4 q+1}-2}{2 y \left(a^{2 q}-1\right)^3
   \left(a^{2 q}+1\right)},\ w=-\frac{4 a^{4 q}+2 a^{6 q}-2 a^{8 q}+2 a^{4 q+1}-2}{2 \left(a^{2 q}-1\right)^2 \left(a^{2
   q}+1\right)}.
\end{split}
\end{equation}
}
\end{proof}

\begin{thm}\label{thm2-1}
{Let $M=S^3_{p/1}(4_1)$ for some integer $p\neq0$.} For $a\in Q_M^{-1}(0)\cap D$ as in Proposition \ref{prop1-1}, we denote the representative $\SL_2(\C)$-representation of $\Phi_M^{-1}(a)$ by $\varphi_a$.
Then, the adjoint Reidemeister torsion of $M$ with respect to $\varphi_a$ is computed as
\begin{empheq}[left={\hspace{-13pt}\tau_{\varphi_a}(M)=\empheqlbrace}]{align}
& -\frac{ 4-p+(-2+p)a^2+2pa^4+(2+p)a^6-(4+p)a^8+2pa^{4+p}}{2(a^2-1)^3(1+a^2)}, &\text{if $ \,a\notin\{\pm \sqrt{-1}\}$,} \label{Eq:tor1}\\
& \frac{1}{8} (10+ a p \sqrt{-5}), &\text{if $\,a\in\{\pm \sqrt{-1}\}$,} \label{Eq:tor2}
\end{empheq}
\end{thm}

\begin{thm}\label{thm2-2}
{Let $M=S^3_{1/q}(4_1)$ for some integer $q\neq0$.} For $a\in Q_M^{-1}(0)\cap D$ as in Proposition \ref{prop1-2}, we denote the representative $\SL_2(\C)$-representation of $\Phi_M^{-1}(a)$ by $\varphi_a$.
Then, the adjoint Reidemeister torsion of $M$ with respect to $\varphi_a$ is computed as
\begin{equation}\label{Eq:tor3}
\tau_{\varphi_a}(M)=
 -\frac{a^{6q}(-1+4q+(1-2q)a^{2q}+2(1+a)a^{4q}+(1+2q)a^{6q}-(1+4q)a^{8q})}{2(a^{4q}-1)^3(1-2a^{2q}-a^{4q}-2a^{6q}+a^{8q})}.
\end{equation}
\end{thm}

\begin{proof}[Proof of {Theorems \ref{thm2-1} and \ref{thm2-2}}]
Under the identification of $\mf{g}\cong \C^3$,
we can concretely describe each $\delta^i$ as the matrices according to \eqref{Eq:delta} and the description of $\Phi_M$ {in the proofs of Propositions \ref{prop1-1} and \ref{prop1-2}}.
Applying the $\tau$-chain method in \cite[\S2.1]{Turaev01} to the chain complex $C^*_\varphi(M;\mf{g})$, with the help of a computer program of Mathematica, we can directly obtain the resulting $\tau_{\varphi_a}(M)$.

\end{proof}

\begin{rem}\label{remTor}

\begin{enumerate}[(i)]
\item While this paper deals with the adjoint torsion via adjoint action, the classical {\it Reidemeister torsion} of $M=S^3_{p/q}(4_1)$ with respect to the $\SL_2(\C)$-representation was computed in \cite{kitano2015reidemeister}.
\item When $M=S^3_{p/1}(4_1)$, the torsion $\tau_{\varphi}(M)${, up to sign,}
was computed in \cite{OhtsukiCMP19}. The advantage of Theorem \ref{thm2-1} is that the sign of the torsion is recovered; thus, we can compute the sum of $\tau_\varphi(M)^{n}$'s, as is seen later.
\item \label{rem:a}We can easily check that $\tau_{\varphi_{a^{-1}}}(M)=\tau_{\varphi_a}(M)\in \C^\times$ by using the relation $Q_M(a)=0$ when $a\neq \pm \sqrt{-1}$, and that $Q_M(\pm\sqrt{-1})=0$ with $M=S^3_{p/1}(4_1)$ if and only if $p$ is divisible by $4$.
\item If $p=0$, that is{, if} $M=S^3_{0/1}(4_1)$, then we can similarly compute $\tau_{\varphi_a}(M)$ as $5/4$, $5/4$, $5$, and $5$ with respect to $a=\sqrt{-1}$, $-\sqrt{-1}$, $(1-\sqrt{5})/2$, and $-(1-\sqrt{5})/2$, respectively.
\end{enumerate}
\end{rem}

\subsection{Two key lemmas}\label{seclem}
As preliminaries of the proof of Theorem \ref{thm1}, we prepare two lemmas:
\begin{lem}\label{lem4} Define a polynomial $\kappa_p(x) \in \Z[x]$ by setting
\[ \kappa_p(x) =\left\{
\begin{array}{ll}
(1+x)^2 ,& \mathrm{if} \ \ p= 2m+1,
\\
(1+x^2)^2 ,& \mathrm{if} \ \ p= 4m, \\
1 ,& \mathrm{if} \ \ p= 4m+2, \\
\end{array}
\right.\]
for some $m \in \Z$. 
Then,
$Q_{M}(x)$ with $M=S^3_{p/1}(4_1) $ is divisible by $\kappa_p(x) $,
and the quotient $ Q_{M}(x)/\kappa_p(x)$ has no repeated roots.
On the other hand,
$Q_{M}(x)$ with $M=S^3_{1/q}(4_1) $ is divisible by $(1+x)^2 $,
and the quotient $ Q_{M}(x)/(1+x)^2 $ also has no repeated roots.
\end{lem}
\begin{proof}
The required statement with $|p| \leq 4$ and $|q| \leq 4$ can be directly shown, we may assume $|p| \geq 5 $
and $|q| \geq 5.$
We first focus on the case $M=S^3_{p/1}(4_1) $.
By a computation of $ \frac{\rm d^n \ }{\rm dx^n} (Q_M(x))\mid_{x=b}$ with $b= \pm 1 , \pm \sqrt{-1}$, we
can easily verify the multiplicity of $Q_M(x)$.
To elaborate, if $p=2m+1$, then 
\[Q_M(1)=4,\ \ Q_M(\sqrt{-1})=-2(-1)^m\sqrt{-1},\ \ Q_M(-\sqrt{-1})=2(-1)^m\sqrt{-1}\]
are all nonzero, which implies that $1,\pm\sqrt{-1}$ are not roots of $Q_M(x)$.
Furthermore,
\[Q_M(-1)=Q_M'(-1)=0,\ \ \  Q_M^{(2)}(-1)=-2(-12-p^2)\neq0,\]
indicates that $-1$ is a root of $Q_M(x)$ with multiplicity $2$.
When $p=4m$ or $4m+2$, we can analogously determine the multiplicity of $Q_M(x)$ with $b=\pm1, \pm\sqrt{-1}$.
Thus, $Q_M(x)$ is divisible by $\kappa_p(x)$, and
$Q_M(x)/\kappa_p(x)$ is not divisible by $x \pm 1$ and $x^2+1$.

Next, {suppose $Q_{M}(x) $ has a repeated root $a\in \C$ with $a \neq \pm 1 , \pm \sqrt{-1}$}.
Then, $Q_{M}(a) =0$ and $ Q_{M}'(a) =0,$ which are equivalent to 
\begin{equation}\label{opop1} 1-a^p(a^{-4}+a^{-2}+2+a^{2}-a^{4})+(a^p)^{2} =0,\end{equation}
\begin{equation}\label{opop2}
(p-4)a^{-4}+(p-2)a^{-2}+2p+(p+2)a^{2}-(p+4)a^{4}= - 2pa^{p}. \end{equation}
Applying \eqref{opop2} to \eqref{opop1} to kill the term $a^p$, we equivalently have
\[ (1+a)^2(1+a^2)^2 \bigl(p^2 - 16 +(16-2p^2) a^2-(36 +p^2)a^4 +(16-2p^2)a^6 +(p^2 -16)a^8 \bigr) =0.\]
Since $a^2 \neq \pm 1$, {the last quartic term} equation can be solved as
\[a^2 = \frac{ p^2-8 + 2 \eta p\sqrt{p^2 -15} + \varepsilon \sqrt{(40-3p^2 +2\eta p \sqrt{p^2 -15}) (p^2-24 +2\eta p \sqrt{p^2 -15})}}{2p^2 -32},\]
for some $\varepsilon, \eta \in \{ \pm 1\}.$
Let $F$ be the field extension $\Q(a)$ of degree 8. 
Let us regard \eqref{opop1} as a quadratic equation in $F$ of $a^p$.
Since the discriminant is not zero and $|p|>4$, $a^p $ does not lie in $F$. This is a contradiction. 
In summary,
$Q_M(x)/\kappa_p(x)$ has no repeated roots as required.

On the other hand, if
$M=S^3_{1/q}(4_1) $, we can easily show that $Q_M(x)$ is divisible not by $(1+x)^3$ but by $(1+x)^2$.
Similarly, {suppose  $Q_{M}(x) $  has a repeated root $a\in \C$ with $a \neq \pm 1 $.}
Then, $Q_{M}(a) = Q_{M}'(a) =0$.
We can easily see $Q_{M}(1/a) = Q_{M}'(1/ a) =0$ by reciprocity of $Q_M$.
Thus, we obtain $ ( x^{-4q} Q_{M})' (a) = ( x^{-4q} Q_{M})' (1/a) =0,$ which are equivalent to
\begin{equation}\label{opop3}2(1+a)= (2q+1)a^{2q}+ (-2q+1)a^{-2q}-(4q+1)a^{4q} - (-4q+1)a^{-4q}, \end{equation}
\begin{equation}\label{opop4} 2(1+a^{-1})= (2q+1)a^{-2q}+ (-2q+1)a^{2q}-(4q+1)a^{-4q} - (-4q+1)a^{4q}.\end{equation}
Since $a^{-4q}Q_{M}(a) = 0$ is equivalent to
\begin{equation}\label{opop5} 2(1+a) 2(1+a^{-1})= 4( a^{4q}-a^{2q}-a^{-2q} + a^{-4q}),\end{equation}
the substitution of \eqref{opop3} and \eqref{opop4} into \eqref{opop5} gives the equation
\begin{equation}\label{opop6} (1-b)^2 (1+b)^2 (-1 +2b+b^2+2b^3-b^4+16q^2 -16bq^2+36b^2q^2 -16b^3q^2+16b^4 q^2)=0 ,\end{equation}
where we replace $a^{2q}$ by $b$. If $\omega^{2q}= \pm 1$ and $\omega \in \C$, we can easily check $ Q_M( \omega ) \neq 0$ by definition.
Thus, $a^q$ is a solution of the quartic equation in \eqref{opop6} and does not lie in $\Q$, for any $q \in \Z$.
Let $F /\Q$ be the field extension by the quartic equation. By definition, $F$ does not contain $ a $ and $2+a+a^{-1}$,
which contradicts \eqref{opop5} since $|q| >4$. In summary,
$Q_M(x)/(1+x)^2$ has no repeated roots as required.
\end{proof}
Next, we should mention a slight modification of Jacobi's residue theorem:
\begin{lem}\label{lem5}
Fix $ {\zeta} \in {\{ 0, 2\}}$ and ${\theta} \in \{ 1,2\}$.
{Suppose a polynomial $k(x) \in \Q[x]$ has no repeated roots and $k(0)\neq 0$.}
Take another polynomial $g(x)\in \Q[x] $ such that $\mathrm{deg}(g) \leq \mathrm{deg}(k) - {\theta} {\zeta} -2$.
Then, the following sum is zero:
\begin{equation}\label{opop}
\sum_{a \in k^{-1}(0)} \frac{ (1 +a^{{\theta}})^{ {\zeta} } g (a)}{ \frac{\rm d \ }{ {\rm d}x }((1+x^{{\theta}})^{\zeta} k(x))|_{x=a}} =0.\end{equation}
\end{lem}
\begin{proof} If $ {\zeta}=0$, the statement is Jacobi's residue theorem exactly (see, e.g., \cite[Section 6]{tran2023adjoint}).
Thus, we may suppose ${\zeta}=2 $. Note that the derivative of $(1+x^{{\theta}})^{\zeta} k(x) $
is $ {\zeta} {\theta} x^{\theta-1}(1+x^{{\theta}})^{{\zeta} -1} k(x)+ (1+x^{{\theta}})^{\zeta} k'(x)$.
Hence, the left hand side of \eqref{opop} is computed as
$\sum_{a \in k^{-1}(0)} g (a)/ k'(a) $, which is equal to zero by the residue theorem.
\end{proof}

\subsection{Proof of Theorem \ref{thm1} with $n=-1$}\label{subsecProof}

We suppose $n=-1$ and give the proof of Theorem \ref{thm1}.
Recall the fact that $M=S^3_{p/1}(4_1)$ and $M=S^3_{1/q}(4_1)$ are hyperbolic if and only if $|p|\geq 5$ and $|q|\geq2$, respectively{; see, e.g., Theorem 4.7 of \cite{Thurston}}.

First, we focus on the case where $p\geq 5$, $M=S^3_{p/1}(4_1)$, and $p$ is not divisible by $4$.
From the definition of $Q_M(x)$ and Theorem \ref{thm2-1}, we can easily verify
\begin{equation}\label{Eq:proof2}
\frac{1}{\tau_{\varphi_a}(M)}= \frac{2(1-a^2)^3( 1+a^2)a^{p-5} }{Q_M ' (a)}
\qquad \textrm{for any } a \in \left(Q_M^{-1}(0)\cap D\right)\setminus \{\pm \sqrt{-1}\}.
\end{equation}
If $p -2$ is divisible by $4$, we replace $g(x)$ and $k(x)$ by $2(1-x^2)^3( 1+x^2)x^{p-5} $ and $Q_M(x)$, respectively.
Then, Lemma \ref{lem5} with $\zeta=0$ deduces to the required conclusion as
\begin{equation}\label{Eq:proof4} 0 = \sum_{ a \in Q_M^{-1} (0) } \frac{g(a)}{Q_M'(a)} =\sum_{ a \in Q_M^{-1} (0) } \frac{1}{\tau_{\varphi_a}(M)}= 2 \sum_{ a \in Q_M^{-1} (0) \cap D } \frac{1}{\tau_{\varphi_a}(M)} = \sum_{ \varphi \in R_{G}^{\rm irr} } \frac{2}{\tau_{\varphi}(M)}.\end{equation}
Here, the second, third, and fourth equalities immediately follow from \eqref{Eq:proof2}, Remark \ref{remTor} (iii), and Proposition \ref{prop1-1}, respectively.
Meanwhile, when $p -1$ is divisible by $2$, we replace $g(x)$ and $k(x)$ by $2(x-1)( x^4 -1 )x^{p-5} $ and $Q_M(x)/(1+x)^2$, respectively.
Then, we can readily show similar equalities to \eqref{Eq:proof4}.

We further discuss the case of $p/4 \in \Z$. By Lemma \ref{lem4},
$ Q_{M}(x)/(1+x^2) $ lies in $\Z[x]$, and has no double roots.
We let $g(x)$ and $k(x)$ be $2(1-x^2)^3 x^{p-5} $ and $Q_M(x)/(1+x^2)$, respectively.
By Lemma \ref{lem5} with $\zeta=1$ and $\theta =2$, we have
\[ \begin{split}0 &=
\sum_{ a \in k^{-1} (0) } \frac{g(a)}{k'(a)} =
\frac{g( \sqrt{-1})}{k'( \sqrt{-1})}+ \frac{g(- \sqrt{-1})}{k'(- \sqrt{-1})}+\sum_{ a \in Q_M^{-1} (0) \cap D \setminus \{ \pm \sqrt{-1}\} } \frac{2}{\tau_{\varphi_a}(M)}\\
&= \frac{32 \sqrt{-1}}{20-p^2} + \sum_{ a \in Q_M^{-1} (0) \cap D \setminus \{ \pm \sqrt{-1}\} } \frac{2}{\tau_{\varphi_a}(M)} \\
&= \frac{2}{\tau_{\varphi_{\sqrt{-1}}}(M)} +\frac{2}{\tau_{\varphi_{-\sqrt{-1}}}(M)} + \sum_{ a \in Q_M^{-1} (0) \cap D \setminus \{ \pm \sqrt{-1}\} } \frac{2}{\tau_{\varphi_a}(M)} \\
&= \sum_{ a \in Q_M^{-1} (0) \cap D } \frac{2}{\tau_{\varphi_a}(M)}
= \sum_{ \varphi \in R_{G}^{\rm irr} } \frac{2}{\tau_{\varphi}(M)}, 
\end{split}\]
which is the required vanishing identity. Here, the second, fourth, and sixth equalities follow from \eqref{Eq:proof2}, Theorem \ref{thm2-1}, and Proposition \ref{prop1-1}, respectively.

Next, we focus on the case of $q\geq 2$ and $M=S^3_{1/q}(4_1)$.
Similarly to \eqref{Eq:proof2}, we can show
\begin{equation}\label{Eq:proof3}
\frac{1}{\tau_{\varphi_a}(M)}= \frac{2 (a^{4 q}-1)^3 (a^{4 q}-(a^2+a+1) a^{2 q-1}+1) }{
\frac{\rm d \ }{ {\rm d}x }(x^{4q+1} Q_M (x))|_{x=a}}
\qquad \textrm{for any } a \in Q_M^{-1}(0)\cap D .
\end{equation}
By a Euclidean Algorithm,
we can choose a polynomial $h(x) \in \Q[x]$ such that
\[ 2 (x^{4 q}-1)^3 (x^{4 q}-(x^2+x+1) x^{2 q-1}+1) \equiv x^{4q+1} h (x) \qquad ({\rm modulo} \ Q_M(x)), \]
and $\mathrm{deg}h(x ) < 8q -2$.
Recall from Lemma \ref{lem4} that $Q_M(x) $ is divisible by $(1+x)^2 $; thus so is $h(x)$.
In summary, we can define polynomials $g(x)$ and $k(x)$ to be $h(x)/ (1+x)^2$ and $Q_M(x)/(1+x)^2$, respectively.
Then, Lemma \ref{lem5} with $\zeta=2$ and $\theta =1$ readily leads to the same equalities as \eqref{Eq:proof4}.

The proof of the cases of $p\leq-5$ and $q \leq -2$ can be shown in the same manner; so we here do not carry out the detailed proof.

Finally, in the remaining cases of $|p|\leq4$ for $M=S^3_{p/1}(4_1)$, 
we can obtain the following by a direct calculation:
\[\begin{split}
\sum_{\varphi\in R^{\rm irr}_G(M)}\frac{1}{\tau_{\varphi}(M)}&=
\begin{cases}
2, & \text{ if }p\in \{0,\pm1,\pm2,\pm3\},\\
8, & \text{ if }p\in \{\pm4\}.
\end{cases}
\end{split}\]
For example, we now discuss the detail in the case $p=4$ for $M=S^3_{p/1}(4_1)$.
The roots of $Q_M(x)=x^2 + 2 x^4 + x^6=0$ are $x=\pm \sqrt{-1}$.
By Theorem \ref{thm2-1}, we have $\tau_{\varphi_{\sqrt{-1}}}(M)= (5-2 \sqrt{5})/4$ and $\tau_{\varphi_{-\sqrt{-1}}}(M)=(5+2 \sqrt{5})/4$, leading to $\sum_{\varphi\in R^{\rm irr}_G(M)}\tau_{\varphi}(M)^{-1}=8$.
Similarly, the computations in the other cases run well.

\section{Surgeries on the $5_2$-knot}\label{sec52}
We discuss Conjecture \ref{conj1} in the case of $M=S^3_{1/q}(K)$, when $K$ is the $5_2$-knot and $|q|\geq 3$.
Since the outline of the discussion in this section is almost the same as that in Section \ref{secProof}, we now roughly describe the discussion.

As in \eqref{Eq:pres}, {according to \cite{NosakaJMSCT22},} the fundamental group $\pi_1(S^3_{1/q}(5_2))$ is known to be presented as
\[
\pi_1(M)\cong\langle x_1,x_2,{x_3},{x_4}\, |\,
{x_3}x_1^2x_2^{-1}{x_3}^{-1}x_1^{-2},
{x_3}x_2^{-1}{x_3}^{-1}x_1^{-1}x_2,
{x_3}[x_1^2,x_2^{-1}]^q,
{x_4}[x_1^2,x_2^{-1}]^{-1}
\rangle
.
\]
Recall the free groups $F$, $P$, and the homomorphism $\psi$ in Section \ref{subsecPres}.
Let $W\in P*F$ be
\[\rho_1\cdot x_1^2\rho_2x_1^{-2} \cdot (x_1^2x_2^{-1}x_1^{-1})\rho_1^{-1}(x_1^2x_2^{-1}x_1^{-1})^{-1} \cdot (x_1^2x_2^{-1}x_1^{-2}x_2)\rho_2^{-1}(x_1^2x_2^{-1}x_1^{-2}x_2)^{-1} \cdot \rho_4^{-1} \cdot x_4\rho_3x_4^{-1} \cdot \rho_4 \cdot \rho_3^{-1}.\]
Then, each $\delta^*$ can be written as in \eqref{Eq:delta} according to \cite[\S3.1]{NosakaJMSCT22}.
Let $Q_M(x)$ be the polynomial of the form
\[Q_M(x)\coloneqq 1-x-2x^{2q}-x^{4q-1} -2x^{4q}+x^{6q-1}+x^{8q}-2x^{10q-1}-x^{10q}-2x^{12q-1}-x^{14q-2}+x^{14q-1}.
\]
The same statement in Proposition \ref{prop1-2} holds for $M=S^3_{1/q}(5_2)$ and $q\neq 0$,
namely, $R^{\rm irr}_G(M)$ is bijective to $Q_M^{-1}(0)\cap D$.
For $a\in Q_M^{-1}(0)\cap D$, let us denote the representative $\SL_2(\C)$-representation of $\Phi_M^{-1}(a)$ by $\varphi_a$ as in Theorem \ref{thm2-1}. Then, the adjoint torsion $\tau_{\varphi_a}(M)$ is computed as
\[ \tau_{\varphi_a}(M)= -{P(a)}/{(2a^2(a^2-1)^4})\] with the help of a computer program of Mathematica.
Here, $P(a)\in \Z[a]$ is a polynomial defined by setting
\[\begin{split}
P(a)=& 1-2 q+a
(28 q+2)+a^2 (3-42 q)+a^3 (36 q-8)+a^4 (2-20 q)\\
&+a^{2q}\left((4 q-1) a^{-1}+18 q-3 +(3-32 q) a+(4-54 q) a^{2}-2 a^{3}+(8 q-1)
a^{4}\right) \\
&+a^{4q}\left((1-4 q) a^{-1}-10 q +(-8 q-3) a+(38 q-4) a^{2}+(5-34 q) a^{3}+(1-10 q) a^{4}\right) \\
&+a^{6q}\left((10 q-1) a^{-1}+(18 q+2) a+(7-56 q) a+(74 q-8) a^{2}+10 q
a^{3}\right)\\
&+a^{8q}\left((14 q-3) a^{-1} +18 q +(9-76 q) a-3 a^{2}+(16 q-3) a^{3}\right)\\
&+a^{10q}\left((24
q-2) a^{-1} +(1-10 q) a+(2-52 q) a+(-18 q-1) a^{2}\right)\\
&+a^{12q}\left((4 q-1) a^{-2}+8 q a^{-1}+5-62 q+(56 q-6) a+(2-6 q) a^{2}\right).\\
\end{split}\]

In addition, when $G=\SL_2(\C)$, $M=S^3_{1/q}(5_2)$, and $n=-1$,
we can show that Conjecture \ref{conj1} is true for any integers $q\neq0$.
The proof can be shown in the same fashion as Section \ref{subsecProof}.
However, the concrete substitutions of $g(x)$ and $k(x)$ into Lemma \ref{lem5} are slightly complicated.
For this reason, we do not go into detailed proof in this paper.

Incidentally, we give comments on the case $M=S^3_{p/1}(5_2)$ with $p\in \Z$.
With the help of a computer program, we can similarly obtain the polynomial $Q_M(x)$ and
determine the associated torsions $\tau_{\varphi} (M)$. However, the resulting computation of $\tau_{\varphi} (M)$ is more intricate; we do not describe the details.
More generally, to show Conjecture \ref{conj1} with $M=S^3_{p/q}(K)$ for other (twist) knots $K$,
we might need other ideas. This is a subject for future analysis.

\section{The conjecture with $n>0$}\label{secother}
We end this paper by discussing Conjecture \ref{conj1} with $n>0$.
Hereafter, we assume that $G=\mathrm{SL}_2(\mathbb{C})$, $M$ is a closed $3$-manifold, and $R^{\rm irr}_{G}(M)$ is of finite order as above.

First, it is almost obvious that the sum \eqref{key1} is a real number: precisely,
\begin{prop}\label{prop4}
Let $n \in \Z$.
The imaginary part of the sum $\sum_{\varphi \in R^{\rm irr}_{G}(M) } \tau_{\varphi}{(M)}^n $ is zero.
\end{prop}
\begin{proof}
For a homomorphism $\varphi: \pi_1(M) \ra G$, we denote by $\bar{\varphi}$ the conjugate representation.
Then, $ \tau_{ \bar{\varphi}}{(M)} = \overline{\tau_{\varphi}{(M)}} $ by definition.
Since we can select representatives $ \varphi_1, \dots , \varphi_m ,\overline{\varphi_1}, \dots , \overline{\varphi_m}, \eta_1 ,\dots, \eta_n $ of $ R^{\rm irr}_{G }(M) $ such that $ [\eta_i]= [\overline{\eta_i}] \in R^{\rm irr}_{G }(M)$, the imaginary part is zero as required.
\end{proof}
Furthermore, we will discuss the rationality of the sum \eqref{key1}, with $G=\SL_2(\C)$.
For a subfield $F \subset \C$, let $R_{\mathrm{SL}_2(F)}^{\rm irr}(M) $ be the set of the conjugacy classes of all irreducible representations $\pi_1(M) \ra \mathrm{SL}_2(F) .$
\begin{prop}\label{prop5}
Let $F/\Q$ be a Galois extension with embedding $F \hookrightarrow \C.$
Suppose that 
the inclusion $R_{\mathrm{SL}_2(F)}^{\rm irr}(M) \subset R_{\mathrm{SL}_2(\C)}^{\rm irr}(M) $ is bijective as a finite set, and is closed under the Galois action of $\mathrm{Gal}(F/\Q)$.
Then, for any $n\in \Z$, the sum $\sum_{\varphi \in R^{\rm irr}_G(M)} \tau_{\varphi}{(M)}^n $ is a rational number. 
\end{prop}
\begin{proof} By definition, $\tau_{\varphi}(M) \in F^{\times} $, and the map $\tau_{\bullet}(M): R_{\mathrm{SL}_2(F)}^{\rm irr}(M) \ra F^{\times} $ is $ \mathrm{Gal}(F/\Q) $-equivariant.
Thus, the sum lies in the invariant part $F^{\mathrm{Gal}(F/\Q)}$.
Hence, by $F^{\mathrm{Gal}(F/\Q)}=\Q $, the sum \eqref{key1} lies in $\Q$ as desired.
\end{proof}
\begin{cor}\label{prop6} Suppose that $p$ is even, and is relatively prime to $q$. 
Let $K$ be a twist knot, and $M$ be $S^3_{p/q}(K).$
Then, for any $n\in \Z$, the sum $\sum_{\varphi \in R^{\rm irr}_G(M)} \tau_{\varphi}{(M)}^n $ is a rational number.
\end{cor}
\begin{proof} As is shown in \cite[Section 2]{nosaka2022reciprocity}, there is a Galois extension $F/\Q$ satisfying the condition in Proposition \ref{prop5}.
\end{proof}
Meanwhile, the integrality of the sum \eqref{key1} with $ M =S^3_{p/q}(K)$ remains a future problem.
When $K $ is either $4_1$- or $5_2$-knot,
we know the resulting computation of $\tau_{\varphi}(M) $ by Theorem \ref{thm2-1}, Theorem \ref{thm2-2}, and Section \ref{sec52}.
Accordingly, it is not so hard to check numerically the conjecture from the computation of $\tau_{\varphi}(M) $ for some small $p,q$.

However, we give the proof of the conjecture multiplied by $2^{2n+1}$ with $ M =S^3_{2m/1}(4_1)$. Precisely,
\begin{prop}\label{prop4}
As in Theorem \ref{thm1}, let $ M =S^3_{2m/1}(4_1)$.
If $n>1$, then the 8-fold sum $2\sum_{\varphi \in R^{\rm irr}_G(M)} ( 8 \tau_{\varphi}(M))^{n}$
is an integer.
\end{prop}
\begin{proof}
Since the proof with $|2m| \leq 4$ is a direct computation, we may suppose $|2m| >4 $.

We will discuss integrality.
{We can easily verify some integral polynomials $ h(x),k(x) \in \Z[x]$ such that
\begin{equation}\label{kkk9} \frac{Q_M(x)}{(1-x^2)^3 }= h(x)+\frac{ m^2-2m+3}{1-x^2}+\frac{ 4m}{(1-x^2)^2}+\frac{ 4}{(1-x^2)^3} , \ \ \ \frac{Q_M(x)}{(1+x^2)^2 }= k(x)+\frac{ 2(1+(-1)^{m-1})}{1+x^2} . \end{equation}
In general, as is known as Newton's formula, the sum $\sum_{ b \in Q_M^{-1} (0) } b^n $ is an integer.
Thus, the sums $\sum_{\alpha \in Q_{M}^{-1}(0) \setminus \{\pm \sqrt{-1} \}  } h(\alpha)^n$ and $\sum_{\alpha \in Q_{M}^{-1}(0)  \setminus \{\pm \sqrt{-1} \} } k(\alpha)^n$ are integers;
Therefore, by \eqref{kkk9}, we can easily check that $\sum_{\alpha \in Q_{M}^{-1}(0)} 4^n (1 - \alpha^2)^{-n}$ and $\sum_{\alpha \in Q_{M}^{-1}(0)} 4^n (1 + \alpha^2)^{-n}$ are integers
by induction on $n$.

Let us complete the proof. Recall from Theorem \ref{thm2-1} the resulting computation of the torsion $\tau_{\varphi_a}(M) $ for $a \in Q_M^{-1} (0) \cap D$;
by the Euclidean Algorithm, we can show
\begin{equation}\label{25} 2 \tau_{\varphi_a}(M) =\ell(a)+ \frac{6+2m-2m^2-m^3}{1-a^2}+ \frac{-6+6m+2m^2}{(1-a^2)^2}+ \frac{-4m}{(1-a^2)^3} + \frac{m( -1+ (-1)^m ) }{2(1+a^2)}\end{equation}
for some $ \ell (a) \in \Z[a]$.
Notice from by Proposition \ref{prop1-1} that
\[ 2 \sum_{ \varphi \in R^{\rm irr}_G(M) } (8 \tau_{\varphi}(M))^n = -(8 \tau_{\varphi_{\sqrt{-1}}}(M))^n -(8 \tau_{\varphi_{\sqrt{-1}}}(M))^n +\sum_{ a \in Q_{M}^{-1}(0) \setminus \{\pm \sqrt{-1} \}  } (8 \tau_{\varphi_a}(M))^n.\]
The first and second terms are integers by an elementary discussion.
The last one is a sum of the above sums $ \sum_{\alpha \in Q_{M}^{-1}(0) } 8^n (1 \pm \alpha^2)^{ s}$, where $ s \leq 3n$.} Hence, the sum $2 \sum_{ \varphi \in R^{\rm irr}_G(M) } (8 \tau_{\varphi}(M))^n $ is an integer as required.
\end{proof}

\subsection*{Acknowledgments}
{The author is deeply grateful to the referee for providing insightful comments and constructive feedback.}
She sincerely expresses her gratitude to Takefumi Nosaka for encouragement and useful advice.
She thanks Teruaki Kitano, Yuta Nozaki, and Yoshikazu Yamaguchi for giving valuable comments.
She is also greatly indebted to Seokbeom Yoon for carefully reading the paper and his insightful discussions.

\bibliographystyle{abbrv}
\bibliography{references.bib}


%
%
%

\end{document}